\documentclass[11pt,noinfoline]{imsart}  

\usepackage{amsmath}
\usepackage{amssymb}
\usepackage{amsthm}
\usepackage{amscd,enumerate}
\usepackage{amsfonts}
\usepackage{caption}
\usepackage[mathscr]{euscript}
\usepackage{graphicx}%
\usepackage{fancyhdr, geometry}
\usepackage{url,  color, verbatim}
\usepackage[normalem]{ulem}
\usepackage[colorlinks]{hyperref} 
\usepackage{mathtools}
\usepackage{tikz}
\usetikzlibrary{matrix}
\usepackage{tabularx}
\usepackage{booktabs}

\numberwithin{equation}{section}

\theoremstyle{plain} 
\newtheorem{thm}{Theorem}[section]
\newtheorem{corollary}[thm]{Corollary}

\newtheorem{cor}[thm]{Corollary}

\newtheorem{lm}[thm]{Lemma}

\newtheorem{prop}[thm]{Proposition}
\theoremstyle{definition}

\newtheorem{remark}[thm]{Remark}
\newtheorem{rmk}[thm]{Remark}

\newcommand{\umodel}{\mathcal{I}(n,D,p)}
\newcommand{\rmi}{\mathcal{I}}

\newcommand{\prob}[1]{\mathbb{P}\left[#1\right]}
\newcommand{\expect}[1]{\mathbb{E}\left[#1\right]}

\newcommand{\bigoh}[1]{\mathrm{O}\left(#1\right)}

\newcommand{\kring}{K[x_1,\ldots,x_n]}

\newcommand{\intvecs}{\mathbb{Z}^n_{\geq 0}}
\newcommand{\bbZ}{\mathbb{Z}}

\DeclareMathOperator{\supp}{supp}
\DeclareMathOperator{\adeg}{arith-deg}

\begin{document}
\title{Asymptotic Degree of Random Monomial Ideals}
\runtitle{Asymptotic Degree}
\author{Lily Silverstein}
\address{California State Polytechnic University, Pomona\\ \emph{\texttt{lsilverstein@cpp.edu}}}

\author{Dane Wilburne}
\address{The MITRE Corporation\thanks{{\bf Approved for Public Release; Distribution Unlimited. Public Release Case Number 20-0202. The second author's affiliation with The MITRE Corporation is provided for identification purposes only, and is not intended to convey or imply MITRE's concurrence with, or support for, the positions, opinions, or viewpoints expressed by the author. $\copyright$ 2020 The MITRE Corporation. ALL RIGHTS RESERVED.}}\\\emph{\texttt{dwilburne@mitre.org}}}

\author{Jay Yang}
\address{University of Minnesota - Twin Cities\\\emph{\texttt{jkyang@umn.edu}}}

\maketitle

\section*{Abstract}

One of the fundamental invariants connecting algebra and geometry is the degree of an ideal. In this paper we derive the probabilistic behavior of degree with respect to the versatile Erd\H{o}s-R\'enyi-type model for random monomial ideals defined in \cite{rmi}. 

 We study the staircase structure associated to a monomial ideal, and show that in the random case the shape of the staircase diagram is approximately hyperbolic, and this behavior is robust across several random models. Since the discrete volume under this staircase is related to the {summatory higher-order divisor function} studied in number theory, we use this connection and our control over the shape of the staircase diagram to derive the asymptotic degree of a random monomial ideal. 

Another way to compute the degree of a monomial ideal is with a standard pair decomposition. This paper derives bounds on the number of standard pairs of a random monomial ideal indexed by any subset of the ring variables. The standard pairs indexed by maximal subsets give a count of degree, as well as being a more nuanced invariant of the random monomial ideal. 

\medskip
\section{Introduction}

One way to understand a complicated class of mathematical objects is to study random instances. This approach has proven to be particularly fruitful in combinatorics, where, for example, the theory of random graphs has a long and rich history (e.g., \cite{Gilbert,ErdosRenyi,AlonSpencer}).  There is also a robust literature on the properties of random simplicial complexes (e.g., \cite{LinialMeshulam, costafarber,KahleClique,bobrowskiCOMPLEXES,babsonFUNDAMENTAL}). On the algebraic side, the study of random groups has received much attention (e.g., \cite{Gromov}). Work now considered classical includes the study of random varieties, defined by random coefficients on a fixed Newton polytope support, as in \cite{kac, Kouchnirenko1976, sturmfelssurvey} and the references therein. The field of {smooth analysis} studies how algorithmic performance varies under random perturbations of the problem input. Contributions to algebraic geometry using smooth analysis include \cite{BePa08a} and \cite{burgissercucker}. More recently, Ein, Erman, and Lazarsfeld \cite{ein+erman+lazarsfeld} (see also \cite{EinLazarsfeld,EIN+quick}), studied the Betti numbers of modules defined by uniformly random Boij-S\"oderberg coefficients \cite{eisenbud+schreyer}.

Now, a new program is underway in the field of commutative algebra, centering on the study of random monomial ideals \cite{silverstein2019probability,DWthesis,rmi,de2019average,erman2018random}. Monomial ideals are the simplest polynomial ideals, yet they are general enough to capture the entire range of possible values for many important invariants of ideals. In \cite{rmi}, the authors described thresholds for the dimension of Erd\H{o}s-R\'enyi-type random monomial ideals; we extend those results by describing in detail what happens along the phase transitions. In \cite{de2019average}, the authors explained the asymptotic and threshold behavior of projective dimension, genericity, and certain simplicial resolutions for random monomial ideals. In 2018, Erman and Yang studied the Stanley-Reisner ideals associated to random flag complexes \cite{erman2018random}, using the probabilistic method to exhibit concrete examples of the asymptotic behavior of syzygies described in \cite{eisenbud+schreyer}.


In this paper, we advance this program by describing the asymptotic degree of Erd\H{o}s-R\'enyi-type random monomial ideals.
 Specifically, a random monomial ideal $\rmi$ in the polynomial ring $R = K[x_1,\ldots,x_n]$ is produced by randomly selecting its generators independently, with probability $p\in (0,1)$ each, from the set of all monomials in $R$ of positive degree no more than $D$. The resulting distribution on generating sets induces a probability distribution on the set of all monomial ideals in the ring $R$, which we denote by $\umodel$. In this paper, the asymptotic behavior of $\rmi\sim\umodel$ will always refer to the case where $n$ is fixed, $D\to\infty$, and $p=p(D)$ is a function of $D$.
 
 In the specific case where $n=2$ and the dimension of the ideal is $0$, we can also view this as giving a random partition by taking the complement of the staircase diagram. Random partitions and tableaux have been extensively studied, for instance in \cite{Vershik-Kerov,Romik-Piotr}, and yield similar pictures to those we will discuss later. However, there is not a clear relationship between our model of a random monomial ideal and any of the studied random partition models.

If $I$ is a zero-dimensional monomial ideal of $R$, its degree equals the number of standard monomials of $R/I$, which equals the number of integer lattice points under the ``staircase" defined by the generators of $I$ (see, e.g., \cite{CoxLittleOshea,MillerSturmfels} and Figure~\ref{fig:staircase}).
The key observation in this paper is contained in Proposition~\ref{prop:hyperbola}, which
establishes that for $n$ fixed and $D\gg0$, the staircase of a random monomial ideal $\rmi\sim\umodel$ is ``approximately hyperbolic.'' More specifically, as $D$ tends to infinity, the multidegrees $(\alpha_1,\ldots,\alpha_n)$ of the minimal generators and first syzygies of $\rmi$ (the ``outside" and ``inside" corners of the staircase, respectively) will all be contained, with probability one, in regions bounded by hyperboloids of the form
\begin{equation}
\prod_{i=1}^n(\alpha_i+1)=d(D).
\end{equation} 
See Remark \ref{remark:hyperbolas} for appropriate choices of the functions $d(D)$.

When $\rmi$ is zero-dimensional, its degree is bounded by the number of lattice points under these hyperboloids. 
Let $Z(n,d)$ denote the number of integer lattice points in the region bounded above by $\prod_{i=1}^n(\alpha_i+1)=d$, and below by the coordinate axes. That is,
\[
Z(n,d) = \#\{\alpha\in\mathbb{Z}_{\ge 0}^n: \prod_{i=1}^n(\alpha_i+1) \le d\}.
\]  

The value $Z(n,d)$ is also given by the classical number theory problem of computing the \emph{summatory higher-order divisor function}, which is equivalent, as a function of $d$, to the quantity
$\#\{\alpha\in\mathbb{Z}_{>0}^n: \prod_{i=1}^n \alpha_i\le d\}$.
By standard arguments from multiplicative number theory (see, e.g., \cite[Theorem 7.6]{nathanson2008elementary}), the summatory higher-order divisor function is asymptotically equal to
\begin{equation}\label{eq:number-theory-asymptotic}
Z(n,d)=\frac{d(\log d)^{n-1}}{(n-1)!}+\bigoh{c(\log d)^{n-2}}.
\end{equation}
We use this to obtain new asymptotic results in random commutative algebra; such as the following:

\begin{thm}\label{thm:degree-intro-theorem}
  Let $\rmi\sim\umodel$, and suppose $p=D^{-k}$ for $k\in (0,n)$, not an integer. Let $s=\lfloor k\rfloor$. Then there exist constants $C_1,C_2>0$ such that asymptotically almost surely as $D\to\infty$,
  \begin{enumerate} 
  	\item $Z(n-s,D^{k-s-\epsilon})<\deg(\rmi)< Z(n-s,D^{k-s+\epsilon})$, and
  	\item $C_1 D^{k-s-\epsilon}(\log D)^{n-1}<	\deg(\rmi)	<		C_2 D^{k-s+\epsilon}(\log D)^{n-1}$ .
  \end{enumerate}
\end{thm}
\begin{proof}
	The first statement is Theorem \ref{thm:general-deg-thm} with the specific choices $f_s=D^{k-s-\epsilon}$, $h_s=D^{k-s+\epsilon}$. The second statement follows from using Equation \ref{eq:number-theory-asymptotic} to give the asymptotics of the first statement. 
\end{proof}

Here and throughout the paper we use the term \emph{asymptotically almost surely} or \emph{a.a.s.} to mean that an event occurs with probability $1$ in the limit as $D\rightarrow \infty$. 

\begin{rmk}
  In Theorem \ref{thm:degree-intro-theorem}, we impose the condition $k\not\in\mathbb{Z}$ in order to present the simplest version of the results in Theorem \ref{thm:general-deg-thm}. When $k$ is not an integer, by \cite{rmi} we know the dimension of $\rmi$ is $\lfloor k\rfloor$ a.a.s., and can therefore dispense with the conditional probability that appears in the full statement of Theorem \ref{thm:general-deg-thm}.
\end{rmk}

 When $I$ is positive-dimensional, $\deg(I)$ is no longer equal the number of lattice points under the monomial staircase, but is still determined by staircase combinatorics. The \textit{standard pair decomposition} of a monomial ideal $I$ is a partition of its standard monomials that simultaneously describes its degree and arithmetic degree. Standard pairs were first introduced in \cite{sturmfels1995bounds}, and are useful both in theory and in computational applications (e.g., \cite{hosten-thomas}).
 
 An \emph{admissible pair} of $I$ is a pair $(x^\alpha, S)$, for $x^\alpha$ a monomial of $\kring$ and $S\subseteq\{x_1,\ldots,x_n\}$, such that $\supp(x^\alpha)\cap S =\emptyset$ and every monomial in $x^\alpha\cdot K[S]$ is a standard monomial of $I$. An admissible pair is called a \emph{standard pair} if it is minimal with respect to the partial order given by $(x^\alpha, S) \le (x^\beta, T)$ if $x^\alpha$ divides $x^\beta$ and $\supp(x^{\beta-\alpha})\cup T \subseteq S$. As an abuse of notation, we will also consider pairs of the form $(x^{\alpha},S)$ for $S\subset [n]$ to be a standard/admissible pair when the pair $(x^{\alpha},\left\{x_i | i\in S\right\})$ is a standard/admissible pair.
The \emph{arithmetic degree} of $I$ equals the number of standard pairs of $I$, while its degree equals the number of standard pairs $(x^\alpha,S)$ with $|S|=\dim I$. For an monomial ideal $I$, we denote it's unique minimal generating set as $G(I)$.
 


In Section~\ref{sec:standard-pairs}, we probabilistically bound the number of standard pairs $(x^\alpha, S)$ of a random monomial ideal for each $S\subseteq \{x_1,\ldots,x_n\}$ as follows: 
\begin{thm}
	\label{thm:sp-count-whp-special-case}
	Fix $S$ a subset of the variables $\{x_1,\ldots,x_n\}$, and let $\rmi\sim\umodel$ where $p=D^{-k}$, $k\in (0,n)$. Then there exists a constant $C>0$ such that asymptotically almost surely as $D\to \infty$,
	\[{CZ(t,D^{k-s-\epsilon})<\#\{\text{standard pairs }(x^\alpha, S) \text{ of }\rmi\}<Z(t,D^{k-s-1+\epsilon})}
	.
	\]  
\end{thm}
\begin{proof}
	This is Theorem \ref{thm:sp-count-whp} for the special case $p=D^{-k}$, using the specific choices $f_s=D^{k-s-\epsilon}$ and $h_{s+1}=D^{k-s-1+\epsilon}$. 
\end{proof}

\begin{figure}[h]
	\centering
	\captionsetup{width=.9\linewidth}
	\includegraphics[height=0.6\linewidth]{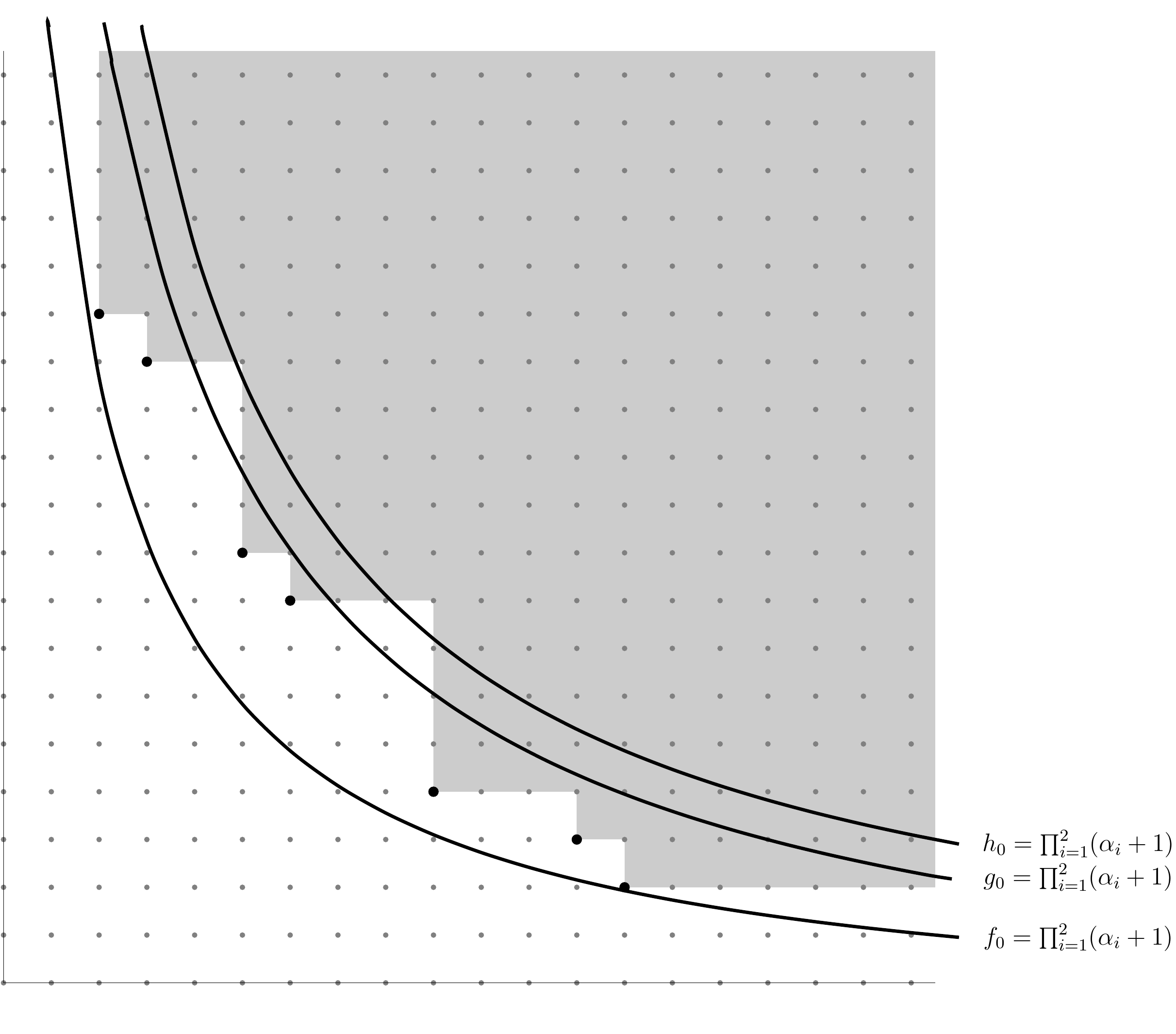}
	\caption{This figure illustrates the role of the functions $f_s$, $g_s$, and $h_s$, as defined in Remark \ref{remark:hyperbolas}, in relation to a random monomial ideal $\rmi\sim\umodel$. For this figure, $s=0$ and $n=2$.}
	\label{fig:staircase}
\end{figure}

\begin{remark}
  \label{remark:hyperbolas}
  For all the results of this paper, it will be convenient to fix the following set of functions which will serve as lower and upper bounds throughout.
Let $\rmi\sim\umodel$ with $D\to\infty$ and fix $n\geq s\geq 0$. 
Then, for $C$ sufficiently small depending only on $n$ and $s$, fix functions  $f_s(D)$, $g_s(D)$, $h_s(D)$ satisfying
    \[pD^{s}f_s(\log f_s)^{t-1}\rightarrow 0,\qquad
  pD^{t}\exp(-CpD^{s}g_s)\rightarrow 0,\qquad
  \mathrm{and}\qquad
  D^{t}\exp(-CpD^{s}h_s)\rightarrow 0.\]

  In particular, as seen in Theorems \ref{thm:degree-intro-theorem} and \ref{thm:sp-count-whp-special-case}, the functions $f_s=D^{-s-\epsilon}/p$ and $g_s=h_s=D^{-s+\epsilon}/p$ satisfy the conditions of the above definition.

These functions describe upper and lower bounds on the multidegrees of generators and syzygies of $\rmi$ in the theorems in this paper as illustrated in Figure~\ref{fig:staircase} for the case $s=0$, $n=2$. As described in Proposition~\ref{prop:hyperbola} and Theorem~\ref{thm:projected-hyperbola}, $f_0$ and $g_0$ provide asymptotic lower and upper bounds on the generators of a random ideal in that asymptotically almost surely, all minimal generators $x_1^{\alpha_1}x_2^{\alpha_2}$ are bounded by  $f_0<(\alpha_1+1)(\alpha_2+1)<g_0$. Visually, this corresponds to the bottom corners of the monomial staircase (in black) lying between two hyperbolic curves.

The function
$h_0$ gives an upper bound to the full staircase in that a.a.s.~every monomial satisfying $(\alpha_1+1)(\alpha_2+1)>h_0$ belongs to $\rmi$. Visually, this corresponds to every lattice point above the upper hyperbola being in the shaded region above the monomial staircase.

\end{remark}

\section{The zero-dimensional case}
When $\rmi$ is zero-dimensional, its standard pairs are exactly the pairs $(x^\alpha,\emptyset)$ where $x^\alpha$ is a standard monomial of $R/\rmi$. In other words, enumerating standard pairs is equivalent to enumerating standard monomials in the zero-dimensional case; this count is also equivalent to the degree of the zero-dimensional ideal.

Proposition~\ref{prop:hyperbola} makes precise the image in Figure~\ref{fig:staircase}. In particular, it shows that the staircase diagram of a monomial ideal is bounded below by $f_0$ and above by $h_0$, with $g_0$ providing a tighter bound on just the generators as opposed to the whole of the staircase diagram. As a remark, Proposition~\ref{prop:hyperbola} uses the same conventions and notation defined in Remark~\ref{remark:hyperbolas}, except that in this section we can make the explicit choice $C=1$.

\begin{prop}\label{prop:hyperbola} Let $\rmi\sim\umodel$ with $D\to\infty$ and $p=p(D)$ a function of $D$. Fix functions 
	$f_0(D)$, $g_0(D)$, $h_0(D)$ satisfying
	\[pf_0(\log f_0)^{n-1}\rightarrow 0,\qquad
	pD^{n}\exp(-pg_0)\rightarrow 0,\qquad
	\mathrm{and}\qquad
	D^{n}\exp(-ph_0)\rightarrow 0.\]
	
  Then,
  \begin{enumerate}
\item  
$\displaystyle\prob{f_0< \prod_{i=1}^n(\alpha_i+1)< g_0\mathrm{\ for\ all\ } x^{\alpha}\in G(\rmi)}\to 1,$ and
\item 
$
\displaystyle\prob{x^{\alpha}\in\rmi \mathrm{\ for\ all\ } x^{\alpha}\mathrm{\ s.t.\ }h_0< \prod_{i=1}^n(\alpha_i+1)\mathrm{\ and\ } |\alpha|\le D}\to 1.$
\end{enumerate}

\end{prop}

\begin{proof}
For the first statement, observe that
$
\prob{x^{\alpha}\in G(\rmi)}=pq^{-1+\prod_{i=1}^n (\alpha_i+1)}.$
Therefore
\begin{align}\label{eq:expected-min-gens}
  \expect{\#\{x^{\alpha}\in G(\rmi):\prod_{i=1}^n (\alpha_i+1)>g_0\}}
  &=\sum_{\substack{\alpha\in\intvecs\\\text{ s.t. }|\alpha|\le D\text { and } \prod (\alpha_i+1)>g_0}}pq^{-1+\prod_{i=1}^n (\alpha_i+1)}\\ \nonumber
  &\le pq^{g_0-1}\sum_{\substack{\alpha\in\intvecs\\\text{ s.t. }|\alpha|\le D}}1\\ \nonumber
  &\le pq^{g_0-1} D^n\\ \nonumber
  &=p(1-p)^{g_0-1} D^n\\ \nonumber
  &\sim pD^{n}\exp{(-pg_0)},
\end{align}
which goes to $0$ as $D \to \infty$ by hypothesis.
Moreover,
\begin{align*}
  \expect{\#\{x^{\alpha}\in G(\rmi):\prod_{i=1}^n (\alpha_i+1)<f_0\}}&\leq\sum_{\substack{\alpha\in\intvecs\\\text{ s.t. }\prod (\alpha_i+1)<f_0}}pq^{-1+\prod_{i=1}^n (\alpha_i+1)}\\
  &\le p\sum_{\substack{\alpha\in\intvecs\\\text{ s.t. } \prod (\alpha_i+1)<f_0}}1\\
  &= p\left(\frac{f_0\left(\log f_0\right)^{n-1}}{(n-1)!}+\bigoh{f_0\left(\log f_0\right)^{n-2}}\right),
\end{align*}
where the last line follows from Equation \ref{eq:number-theory-asymptotic}.  This last expression goes to 0 as $D\to\infty$, again by hypothesis. Thus, asymptotically almost surely, every minimal generator $x^\alpha$ satisfies $f_0< \prod_{i=1}^n(\alpha_i+1)< g_0$.

To prove the second statement, observe that $
\prob{x^{\alpha}\notin\rmi}=q^{-1+\prod_{i=1}^n (\alpha_i+1)},$ and so
\begin{align*}
\expect{\#\{x^{\alpha} \not \in \rmi :  \prod (\alpha_i+1)>h_0  \mbox{  and  } |\alpha| \le D \}}&=\sum_{\substack{\alpha\in\intvecs\\\text{ s.t. }|\alpha|\le D\text { and } \prod (\alpha_i+1)>h_0}} q^{-1+\prod_{i=1}^n (\alpha_i+1)}.
\end{align*}
Using the same estimates as in Equation \ref{eq:expected-min-gens}, this expectation is bounded above by $q^{h_0}D^n\sim D^n\exp(-ph_0)$, which goes to zero by hypothesis.
\end{proof}

If $\rmi$ is zero-dimensional, the previous proposition immediately implies bounds on the degree of $\rmi$. 

\begin{cor}\label{cor:asymptotic-degree-bounds}
  For $\rmi\sim\umodel$, with $D\to\infty$ and $p=p(D)$ satisfying $1/D\ll p\leq 1$, we have
\[
\prob{Z(n, f_0)\leq \deg(\rmi)\leq Z(n, h_0)}\to 1.
\]
\end{cor}

\begin{proof} 
	Whenever $\rmi$ is zero dimensional, its degree equals the number of standard monomials of $R/\rmi$. And by Proposition~\ref{prop:hyperbola}, every monomial $x^\alpha$ with $\prod(\alpha_i+1) < f_0$ is a standard monomial of $R/\rmi$, so there must be at least $Z(n, f_0)$ standard monomials. On the other hand, by the same proposition we know that every monomial $x^\alpha$ with $\prod(\alpha_i+1) > h_0$ is \textit{not} a standard monomial of $R/\rmi$, so there are at most $Z(n,h_0)$ standard monomials. 
	
	For $1/D\ll p \leq 1$, it follows from \cite[Corollary 1.2]{rmi} that $\prob{\dim(R/\rmi)=0}\to1$ and thus, with probability one, $\deg(\rmi)$ is bounded between these two quantities.
\end{proof}

\section{The positive dimensional case}
This next proposition will be necessary to compute the degree of higher dimensional ideals. For any $T\subseteq [n]$, Theorem~\ref{thm:projected-hyperbola} describes probabilistic constraints on the $x^\alpha$ such that $(x^\alpha, T^C)$ is an admissible pair of $\rmi$. Analogously, Proposition~\ref{prop:hyperbola} can be viewed as giving constraints on admissible pairs of the form $(x^\alpha,\emptyset)$. However, Proposition~\ref{prop:hyperbola} gives slightly better bounds than simply substituting $T=\left\{1,\ldots,n\right\}$ into Theorem~\ref{thm:projected-hyperbola}.

\begin{thm}
  \label{thm:projected-hyperbola}
  Let $\rmi\sim\umodel$ with $D\to\infty$. For $T\subseteq [n]$, define $\rmi|_{T}$ to be the ideal in $K[x_i : i\in T]$ given by substituting $x_i\mapsto 1$ for $i\notin T$. Set $t:=|T|$ and $s:=n-t$. Then there exists a constant $C$ depending only on $n$ and $t$ such that, for any fixed functions  $f_{s}(D)$, $g_{s}(D)$, $h_{s}(D)$ as in Remark~\ref{remark:hyperbolas}, the following statements hold:
\begin{enumerate}
\item  $
\displaystyle
\prob{f_s< \prod_{i\in T}(\alpha_i+1)< g_s\mathrm{\ for\ all\ } x^{\alpha}\in G(\rmi|_T)}\to 1,
$ and
\item 
$
\displaystyle
\prob{x^{\alpha}\in\rmi|_T \mathrm{\ for\ all\ } x^{\alpha}\mathrm{\ s.t.\ }h_s< \prod_{i\in T} (\alpha_i+1)\mathrm{\ and\ } |\alpha|\le D}\to 1.
$
\end{enumerate}
\end{thm}

\begin{remark}
  For example, suppose $p=D^{-k}$ and fix any $\epsilon>0$. As in the theorem, fix a subset of the variables $T$ and let $s:= n-|T|$. Then $f_s=D^{k-s-\epsilon}$, $g_s=h_s=D^{k-s+\epsilon}$ are examples of functions satisfying the hypotheses of Theorem~\ref{thm:projected-hyperbola}. This implies that the corners of the staircase diagram for $\rmi|_{T}$ are confined to an arbitrarily narrow strip around the curve $\prod_{i\in T}(\alpha_i+1)=D^{k-s}$. In particular, the admissible pairs of the form $(x^{\alpha},T^{C})$ have $\alpha$ values bounded by $\prod_{i\in T}(\alpha_i+1)\leq D^{k-s+\epsilon}$.
  
  Note that the map from $\rmi\subseteq\kring$ to $\rmi|_T\subseteq K[x_i:i\in T]$ defined in the theorem is equivalent to saturating by the variables in $S$, followed by intersecting with $K[x_i:i\in T]$.
\end{remark}

\begin{proof}

  Fix $\alpha_i$ for $i\in T$. Then \[B(\alpha,T):=\#\left\{\beta : |\beta|\leq D\text{ and }\beta_i\leq \alpha_{i}\text{ for }i\in T\right\}\geq C\prod_{i\in T} (\alpha_i+1) D^{s}.\]
  To show this inequality, we will use the average value of $\sum_{i\notin T}\beta_i$, and particularly the fact that it is less than $D/2$.
  
  \begin{align*}
    B(\alpha,T)&=\sum_{\substack{\gamma\in \bbZ_{\geq 0}^{T}\\ \gamma\leq\alpha}}\#\left\{\beta\in\intvecs : |\beta|\leq D\text{ s.t. } \beta_i=\gamma_i\text{ for } i\in T\right\}\\
    &=\sum_{\substack{\gamma\in \bbZ_{\geq 0}^{T}\\ \gamma\leq\alpha}} \binom{D-\left|\gamma\right|+s}{s}.\\
    \intertext{Now we restrict the summation to those $\gamma$ where $\left|\gamma\right|\leq {D}/{2}$ yielding the following inequality:}
    B(\alpha,T)&\geq \sum_{\substack{\gamma\in \bbZ_{\geq 0}^{T}\\ \gamma\leq\alpha,\, 2\left|\gamma\right|<D}} \binom{D-\left|\gamma\right|+s}{s}\\
    &\geq \sum_{\substack{\gamma\in \bbZ_{\geq 0}^{T}\\ \gamma\leq\alpha,\, 2\left|\gamma\right|<D}} \binom{D/2+s}{s}\\
    &\geq \#\left\{\gamma\in \bbZ_{\geq 0}^{T}\middle| \gamma\leq\alpha,\, 2\left|\gamma\right|<D\right\} \binom{D/2+s}{s}\\
    &\geq \#\left\{\gamma\in \bbZ_{\geq 0}^{T}\middle| \gamma\leq\alpha/2\right\} \binom{D/2+s}{s}\\
    &\geq \frac{1}{2^{|T|}}\prod_{i\in T} (\alpha_i+1) \binom{D/2+s}{s}\\
    &\geq C\prod_{i\in T} (\alpha_i+1) D^{s}.
  \end{align*}

  Now we can proceed to prove the upper bound. The proof proceeds similarly to Proposition~\ref{prop:hyperbola}. On the one hand, when the product of the $\alpha_i$'s is too large, we have
  
\begin{align}\label{eq:expected-min-gens-projected}
  \expect{\#\{x^{\alpha}\in G(\rmi|_{T}):\prod (\alpha_i+1)>g_s\}}
  &\le\sum_{\substack{\alpha\in\bbZ_{\geq 0}^{T}\\\prod (\alpha_i+1)>g_s,\, |\alpha|\le D}}(1-q^{(D-|\alpha|)^s})q^{B(\alpha,T)}\\ \nonumber
&\le q^{CD^sg_s}\sum_{\substack{\alpha\in\bbZ_{\geq 0}^{T}\\\prod (\alpha_i+1)>g_s,\,|\alpha|\le D}}1\\ \nonumber
&\le C_1q^{CD^sg_s} D^{t}\\ \nonumber
&=C_1p(1-p)^{CD^sg_s} D^{t}\\ \nonumber
&\approx C_1pD^{t}\exp{(-CpD^{s}g_s)}\to 0.
\end{align}
On the other hand,
\begin{align*}
  \expect{\#\{x^{\alpha}\in G(\rmi|_{T}):\prod (\alpha_i+1)<f_s\}}
  &\le\sum_{\substack{\alpha\in\bbZ_{\geq 0}^{T}\\\prod(\alpha_i+1)<f_s}}(1-q^{(D-|\alpha|)^s})q^{B(\alpha,T)}\\
  &\le (1-q^{D^s})\sum_{\substack{\alpha\in\bbZ_{\geq 0}^{T}\\\prod(\alpha_i+1)<f_s}}1.\\
  \intertext{Since $(1+x)^{r}\geq 1+xr$ for $x\geq -1$ and $r\geq 1$, using $x=-p$ and $r=D^s$ we get the following inequality:}
  \expect{\#\{x^{\alpha}\in G(\rmi|_{T}):\prod (\alpha_i+1)<f_s\}}&\le pD^{s}\left(\frac{f_s(\log f_s)^{t-1}}{(t-1)!}+\bigoh{f_s(\log f_s)^{t-2}}\right),
\end{align*}
which goes to zero whenever $f_s(\log f_s)^{t-1}\ll D^{-s}/p$. For the second statement, we have

\begin{align*}
\expect{\#\{x^{\alpha} \not \in \rmi|_T :  \prod (\alpha_i+1)>h_s  \mbox{  and  } |\alpha| \le D \}}&=\sum_{\substack{\alpha\in\bbZ_{\geq 0}^T\\|\alpha|\le D,\, \prod (\alpha_i+1)>h_s}}q^{B(\alpha,T)}.
\end{align*}
An estimate similar to the one in Equation \ref{eq:expected-min-gens-projected} shows that this goes to zero whenever  $D^{t}\exp(-CpD^{s}h_s)$ does.
\end{proof}

Since the degree, in any dimension, can be bounded by the number of admissible pairs with support of a particular size, the previous result leads to bounds on the degree of $\rmi\sim\umodel$.
\begin{thm}
  \label{thm:general-deg-thm}
    Fix $s$, $0\leq s\leq n$. For $\rmi\sim\umodel$, if $p=p(D)$ is any function such that $\lim_{D\rightarrow \infty} \prob{\dim(\rmi)=s}> 0$, then
    \[
    \prob{\binom{n}{s} Z(n-s,f_s)<\deg(\rmi)< \binom{n}{s} Z(n-s,h_s) \,\middle|\, \dim(\rmi)=s}\to 1
    \]
   as $D\to\infty$, where $f_s(D)$, $h_s(D)$ are any functions satisfying the hypotheses of Theorem~\ref{thm:projected-hyperbola}.

\end{thm}
\begin{proof}
  Suppose $\dim(\rmi)=s$. Then  \[\deg(\rmi)=\sum_{T\subset [n],\, \left|T\right| = n-s} \deg(\rmi|_{T}).\]
  Since $\lim_{D\rightarrow \infty} \prob{\dim(\rmi)=s}> 0$, for the following events, it suffices to know that the non-conditional probabilities go to $1$. In particular, suppose that $A(\rmi)$ is an event with $\prob{A(\rmi)}\rightarrow 1$. Then consider the following sum:
  \begin{equation}\label{eq:cond-prob-deg}
  \prob{A(\rmi)}=\sum_{k=0}^n\prob{A(\rmi) \mid \dim(\rmi)=k}\prob{\dim(\rmi)=k}.
  \end{equation}
  If $\lim_{D\rightarrow \infty}\prob{\dim(\rmi)=s}>0$, this forces $\prob{A(\rmi) | \dim(\rmi)=s}\rightarrow 1$.

  By Theorem~\ref{thm:projected-hyperbola}, and the same argument as in the zero dimensional case, $\prob{\deg(\rmi|_{T})<Z(n-s,h_s)}\rightarrow 1$. Summing over all projections onto $n-s$ coordinates, this shows $\prob{\deg(\rmi)<\binom{n}{s}Z(n-s, h_s)}\rightarrow 1$.
  Also by Theorem~\ref{thm:projected-hyperbola}, we know that
  \[
  \prob{f_s< \prod_{i\in T}(\alpha_i+1)\mathrm{\ for\ all\ } x^{\alpha}\in G(\rmi|_T)}\to 1,
  \]
  and therefore with probability approaching one, all $x^\alpha$ below this curve correspond to admissible pairs. Asymptotically, this implies $\prob{Z(n-s,f_s)<\deg(\rmi|_{T})}\rightarrow 1$, and therefore \[\prob{\binom{n}{s}Z(n-s,f_s)<\deg(\rmi)}\rightarrow 1.\qedhere\]
\end{proof}

\begin{remark}\label{rmk:deg-is-conditional}
A few words about the conditional nature of the result in Theorem \ref{thm:general-deg-thm}. As seen in Theorem \ref{thm:degree-intro-theorem} in the introduction, particular choices of $p$ can guarantee that only one dimension is observed a.a.s., and thus the conditional probability that appears in the general statement of Theorem \ref{thm:general-deg-thm} can be dispensed with. In one of the most natural settings, the case where $p(D)=D^{-k}$, $\rmi\sim \umodel$ is a.a.s. a particular fixed dimension for any choice of $k$ other than an integer (this is why we restricted $k$ not to be an integer in Theorem \ref{thm:degree-intro-theorem}).

But what happens if we allow $k$ to be integral? This ``boundary case" was worked out in detail by the second author by analyzing the case where $pD^t$ approaches a constant, where $D^{-t}$ is one of the integral thresholds for dimensionality.

\begin{thm}[{\cite[Theorem 3]{DWthesis}}]
	\label{thm:dane-expected-dim} Let $\rmi\sim\rmi(n,D,p)$, let $1\le t\le n$ be an integer and let $c>0$ be a constant.  If $pD^{t}\to c$ as $D\to\infty$, then
	\[
	\lim_{D\to\infty}\expect{\dim \rmi}=t-(1-e^{-c/t!})^{\binom{n}{t}}.
	\]
\end{thm}

	If we consider a particular regime $\rmi(n,D,p=D^{-k})$ with $k$ integral, then there is not a single dimension a.a.s., but it is a boundary case involving \textit{exactly two} dimensions that appear with nonzero probability, distributed as in \ref{thm:dane-expected-dim}. Combining that Theorem with Theorem \ref{thm:general-deg-thm} allows an explicit evaluation of Equation \ref{eq:cond-prob-deg}, which will have exactly two nonzero summands, and provides the ``boundary case" version of Theorem \ref{thm:general-deg-thm} with no conditional probabilities in the final statement.
\end{remark}

In the next section, we take an alternative approach to remove conditioning on dimension, by using standard pair enumerations that are robust across dimension.

\section{Standard pairs}\label{sec:standard-pairs}

To demonstrate the well-behaved nature of standard pair invariants (as opposed to the degree invariant), consider several experimental samples from the $\rmi(3,D,p=D^{-2})$ regime, which is chosen so that $\rmi$ sampled from this distribution will have non-negligible positive probability of being either 1 or 2 dimensional, a.a.s. A small value of $D=65$ is enough to display the relevant behavior.

\begin{table}[h]
\begin{center}
	\begin{tabular*}{0.92\textwidth}{ccrrrr}
	$\rmi\sim\rmi(3,65,p=1/4225)$ &  $\dim \rmi$ & $\deg \rmi$ & $sp_0$ & $sp_1$ & $sp_2$\\
	\toprule
	{\footnotesize $ \langle x_1^{8} x_2^{35} x_3^{5},x_1^{8} x_2^{25} x_3^{11},x_1^{18} x_2^{16} x_3^{16},x_1 x_2^{29} x_3^{31},x_1^{5} x_2^{14} x_3^{40},x_1^{2} x_2^{19} x_3^{40}  
 \rangle $ }
	&  2 &  20 &  2781 &  441 &  20\\
	\midrule
	{\footnotesize$\langle x_1^{33} x_2^{23},x_1^{40} x_2 x_3,x_1^6 x_2^{49} x_3^4,x_1^{21} x_2^6 x_3^5,x_1^{19} x_2^3 x_3^{28},x_1^{11} x_2^{16} x_3^{28},x_1^{13} x_2^2 x_3^{36} \rangle$}
	& 2 & 7 & 14348 & 427 & 7\\
	\midrule
	{\footnotesize $\langle x_1 x_2^{45} x_3,x_1 x_2^{21} x_3^{4},x_1^{14} x_2^{6} x_3^{6},x_1^{38} x_2^{4} x_3^{17},x_1^{2} x_3^{37},x_2^{25} x_3^{39},x_3^{52}  
 \rangle $}
	&  2 &  1 &  8165 &  361 &  1\\
	\midrule
	{\footnotesize$\langle x_1^{50} x_2^{14},x_1^{7} x_2^{41},x_1^{51} x_2^{2} x_3^{4},x_1^{10} x_2^{24} x_3^{4},x_1^{6} x_3^{8},x_2^{27} x_3^{8},x_1^{3} x_2^{14} x_3^{16},x_2^{25} x_3^{40}\rangle$}
	& 1 & 237 & 9184 & 237 & 0\\
	\midrule
	{\footnotesize $ \langle x_1^{12} x_2^{52},x_1^{4} x_2^{16} x_3^{3},x_1^{54} x_2^{6} x_3^{4},x_1^{40} x_2^{11} x_3^{7},x_1^{4} x_3^{10},x_2 x_3^{39}  
 \rangle $ }
	&  1 &  392 &  2790 &  392 &  0\\
	\midrule
	{\footnotesize $ \langle x_1^{30} x_2^{5},x_1^{28} x_2^{22} x_3,x_1^{18} x_2^{22} x_3^{8},x_1^{36} x_2^{3} x_3^{9},x_1^{6} x_2^{31} x_3^{9},x_2^{4} x_3^{13},x_1 x_3^{54}  
 \rangle $ }
	&  1 &  452 &  4181 &  452 &  0\\
	\bottomrule
\end{tabular*}
\caption{A collection of randomly generated ideals showing the sensitivity of the degree at the dimension boundary as compared to the number of standard pairs. Here, $sp_i$ is the number $i$-dimensional standard pairs.}
\label{tab:sp-counts}
\end{center}
\end{table}

Because these parameters were chosen to be a boundary case, we see samples containing both one- and two-dimensional ideals. The degrees of the dimension one versus dimension two ideals are dramatically different. This small example shows that the degree invariant is not well behaved when considering multiply-dimensional sets of ideals. This explains why the asymptotic statements in Section 3 required conditioning on dimensions (see Remark \ref{rmk:deg-is-conditional}).

On the other hand, this small glimpse at the data illustrates that the enumeration of standard pairs, whether counted collectively or even dimension-by-dimension, behaves predictably across dimension borders. In other words, the count of standard pairs of a particular dimension is uncorrelated with the true dimension of the ideal. This is obviously false for degree since the definition of degree relies on counting \textit{only} the standard pairs in a particular dimension. Table \ref{tab:sp-counts} also demonstrates a useful fact we'll use in Section \ref{sec:standard-pairs}, which is that the standard pair count is always zero for standard pairs of dimension greater than the ideal of the dimension. This observation will be important to the hypotheses on Lemma \ref{lem:sp-technical} and Theorem \ref{thm:sp-count-whp}. (See Remark \ref{rmk:dim-and-sps}, preceding those results.)

Our next goal in this section is to prove that for any choices of parameters in $\umodel$, and any choice of a subset $S$ of the variables of the ring, there is a region of lattice points guaranteed, asymptotically almost surely, to be standard standard pairs of the form $(x^\alpha,S)$.

\begin{thm}
  \label{thm:std-pair-whp}
  Let $\rmi\sim\umodel$ with $D\to\infty$. Fix $S\subseteq [n]$, with $s:=|S|$, $T:=[n]\setminus S$, and $t:=|T|=n-s$. Let $p=p(D)\rightarrow 0$ and let functions $f_s, h_{s+1}\rightarrow \infty$ and $f_s\leq D$ be as in Remark~\ref{remark:hyperbolas}.
  Then
  \begin{equation*}
    \prob{\text{for all }\alpha\in L(f_s,h_{s+1}),\, (x^{\alpha},S)\text{ is a standard pair for }\rmi} \rightarrow 1,
  \end{equation*}
  where 
  $\displaystyle L(f_s,h_s):=\left\{\alpha\in \bbZ_{\geq 0}^{T}: \prod_{i\in T} (\alpha_i+1) < f_s \text{ and } (\alpha_i+1)^{t-1}>h_{s+1} \right\}$. 
  
\end{thm}


\begin{figure}[h]
	\centering
	\captionsetup{width=0.9\linewidth}
    \includegraphics[width=0.6\linewidth]{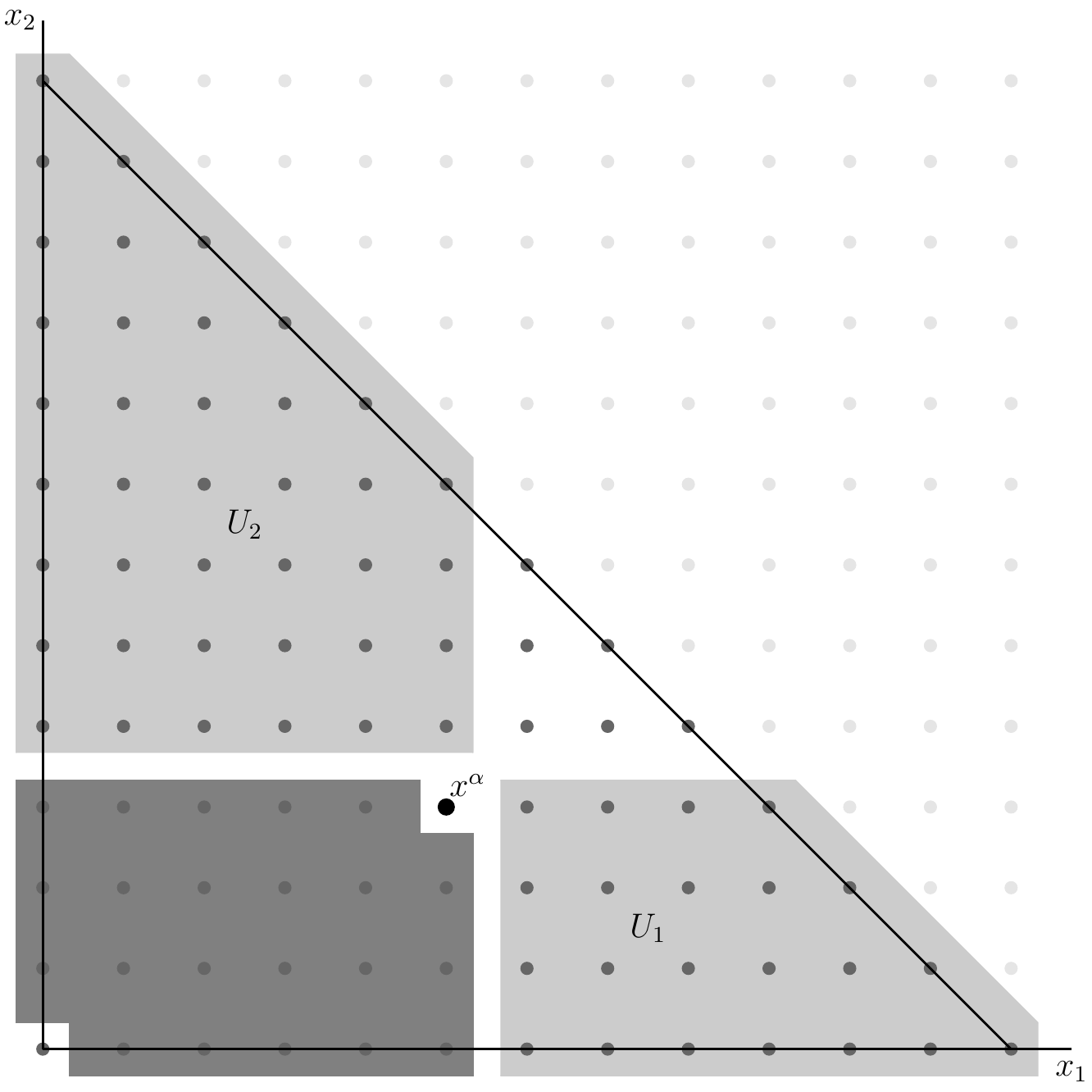}
  \caption{
  	For $n=2$, $D=12$, and $x^\alpha=x_1^5x_2^3$, this figure illustrates the conditions under which $(x^\alpha,\emptyset)$ is a standard pair of $\rmi\sim\umodel$. For $x^\alpha$ to be a standard monomial, no monomial in the dark gray region can be a generator of $\rmi$. Additionally, to insure $x^\alpha$ does not belong to a higher-dimensional standard pair, at least one monomial from each of the light gray regions, $U_1$ and $U_2$, must be a generator of $\rmi$.
  	In the language of Equation \ref{eq:standard-pair-geometry}, at least one generator chosen in region $U_1$ guarantees that $(x_2^3,\{x_1\})$ is not an admissible pair for $\rmi$ (note $x_2^3=x^\alpha|_{x_2}$). Similarly, a generator chosen in region $U_2$ guarantees that $(x_1^5,\{x_2\})$ is not an admissible pair for $\rmi$.}
  \label{fig:trapezoid}
\end{figure}

\begin{proof}
For convenience, we will write $f=f_s$ and $h=h_{s+1}$. Then the statement \[\prob{ \text{for all }\alpha\in L(f,h), (x^{\alpha},S)\text{ is a standard pair for }\rmi} \rightarrow 1,\] is equivalent to the finite set of statements
  \[\prob{\text{for all }\alpha\in L(f,h),\, (x^{\alpha},S)\text{ is an admissible pair for }\rmi} \rightarrow 1,\]
  and for all $i\in T$
  \begin{equation}\prob{\text{for all }\alpha\in L(f,h),\, (x^{\alpha}|_{T\setminus \left\{i\right\}},S\cup \left\{i\right\})\text{ is not an admissible pair for }\rmi} \rightarrow 1.
  \label{eq:standard-pair-geometry}
  \end{equation}
  For the geometric intuition underlying the statements of Equation \ref{eq:standard-pair-geometry}, see Figure \ref{fig:trapezoid}.
  
  \noindent \emph{Part 1:} $\prob{\text{for all }\alpha\in L(f,h),\, (x^{\alpha},S)\text{ is an admissible pair for }\rmi} \rightarrow 1$

  Note $f$ satisfies the same conditions as it would in Theorem~\ref{thm:projected-hyperbola}. As a consequence of that theorem, we have that
  \[\prob{\text{for all }\alpha\in \bbZ^{T}\text{ with }\prod_{j\in T}(\alpha_j+1)<f, (x^{\alpha},S)\text{ is an admissible pair for }\rmi}\rightarrow 1.\] For all $\alpha\in L(f,h)$, we have $\prod_{j\in T}(\alpha_j+1)<f$, and thus
  \[\prob{\text{for all }\alpha\in L(f,h),\, (x^\alpha,S)\text{ is an admissible pair for }\rmi} \rightarrow 1.\]

  \noindent \emph{Part 2 :}  For fixed $i\in T$, $\prob{\text{for all }\alpha\in L(f,h),\, (x^{\alpha}|_{T\setminus \left\{i\right\}},S\cup \left\{i\right\})\text{ is not an admissible pair }\rmi} \rightarrow 1$

  Notice that for a fixed product, the sum is maximized in the case where the product is most asymmetric, and so since $\prod_{i\in T}(\alpha_i+1)< f\leq D$ and $\alpha_i\geq 0$, we have that $|\alpha|< D$. Now let $T'=T\setminus \left\{i\right\}$ and $S'=S\cup \left\{i\right\}$. Again we apply Theorem~\ref{thm:projected-hyperbola}, noting that $h=h_{s+1}$. The theorem implies that

  \[\prob{\text{for all }\alpha\in \bbZ^{T'}\text{ with }\prod_{j\in T'}(\alpha_j+1)>h,\, (x^{\alpha},S')\text{ is not an admissible pair for }\rmi }\rightarrow 1.\]

  For all $\alpha\in L(f,h)$, we have $\prod_{j\in T'}(\alpha_j+1)>h$, and thus
  
  \[\prob{\text{for all }\alpha\in L(f,h),\, (x^{\alpha}|_{T'},S')\text{ is not an admissible pair for } \rmi} \rightarrow 1.\qedhere\]
\end{proof}

\begin{remark}\label{rmk:dim-and-sps}
  Notice that the case where $|S|> \dim \rmi$ the conditions in Remark~\ref{remark:hyperbolas} requires that $f_s$ must be a decreasing to zero function. This ensures the set $L(f_s,h_{s+1})$ is empty for $D$ sufficiently large, and thus Theorem~\ref{thm:std-pair-whp} is vacuously true in these cases, which matches the fact that in these cases,  the dimension results in \cite{rmi} imply that there are no standard pairs with support $S$.

  The next results bound the size of $L(f_{s},h_{s+1})$ to give probabistic estimates on the number of standard pairs in the case where $|S|\leq \dim \rmi$.
\end{remark}

\begin{lm}
  \label{lem:sp-technical}
  Fix $T\subseteq [n]$, with $T\neq \emptyset$ and let $t=|T|$. Given $f,h$ with $f\rightarrow \infty$ and $f^{t-1}\gg h^{t}$. Then there exists some constant $C\leq 1$ such that
  \[|L(f,h)|=CZ(t,f)+O\left(f(\log(f))^{t-2}\right).\]
\end{lm}
\begin{proof}
  The case of $t=1$ causes issues with the remainder of the proof we first prove that case. Notice if $t=1$, the requirement $h\ll f^{t-1}$ implies $h\ll 1$ and so in particular, $h$ is eventually strictly less than $1$. That means for $D$ sufficiently large, we can expand to give $L(f,h)=\left\{a\in\bbZ_{\geq 0} | a+1<f\right\}$. Thus $|L(f,h)|=f-1$, thus the statement is true.
  
  For the remainder of this proof, we will use a result by Davenport\cite{davenport1951lipschitz}, which bounds the difference between the number of lattice points in a region and its volume. For this it will be convenient to consider the following volume:
  \[W_{t}(c,d) := Vol\left(\left\{x\in \mathbb{R}_{\geq 0}^t :  \prod (x_i+c)\leq d\right\}\right).\]
  Equivalently, we may define this volume as $W_{t}(c,d)=Vol\left(\left\{x\in\mathbb{R}^t: x_i\geq c-1\text{ and } \prod (x_i+1)\leq d\right\}\right)$. We use the convention that $W_0(c,d)=1$. 

  Applying the main theorem from Davenport~\cite{davenport1951lipschitz}, as $d\rightarrow \infty$, we get
  \[W_t(1,d)=Z(t,d)\pm O\left(\sum_{i=0}^{t-1}W_{i}(1,d)\right).\]

  Then since $W_0(1,d)=1$ and $W_0(1,d)=d$, we have $W_1(1,d)\gg W_0(1,d)$. Recall that $Z(t,d)=O(d(\log d)^{t-1})$, then by inducting on $i$, we have that $W_i(1,d)\gg W_{i-1}(1,d)$. This allows us to simplify the previous equaiton to
  \[W_t(1,d)=Z(t,d)\pm O\left(W_{t-1}(1,d)\right).\]
  
  Again we can apply the theorem by Davenport, this time to $|L(f,h)|$ to yield the following: \[|L(f,h)|=W_t(\sqrt[t-1]{h},f)\pm O\left(W_{t-1}(\sqrt[t-1]{h},f)\right).\]

  Observe that $W_{t}(c,d)=c^tW_{t}(1,d/c^t)$. So as long as $d/c^t\rightarrow \infty$, we can use the asymptotic behavior of $W_t(1,d)$. In particular, we can expand $W_t(c,d)$ as follows
    \begin{align*}
    W_{t}(c,d)&=c^tW_{t}(1,d/c^t)\\
    &=c^tZ(t,d/c^t)\pm O(c^tW_{t-1}(1,d/c^t))\\
    &=c^t\left(\frac{d/c^t\log^{t-1}(d/c^t)}{(t-1)!}+O(d/c^t\log^{t-2}d/c^t)\right)\pm O(c^tW_{t-1}(1,d/c^t))\\
    &=\frac{d\log^{t-1}(d/c^t)}{(t-1)!}+O(d\log^{t-2}d/c^t)\pm O(c^tW_{t-1}(1,d/c^t))\\
    &=\frac{d\log^{t-1}(d/c^t)}{(t-1)!}+O(d\log^{t-2}d/c^t).
  \end{align*}

  In this case, we use $d=f$ and $c=\sqrt[t-1]{h}$, so $d/c^t=\frac{f}{h^{t/t-1}}$ and thus $d/c^t\rightarrow \infty$. This allows us to apply the previous formula to yield the following:

  \begin{align*}
    |L(f,h)|&=W_t(\sqrt[t-1]{h},f)\pm O(W_{t-1}(\sqrt[t-1]{h},f))\\
    &=\frac{f\log^{t-1}\left(\frac{f}{h^{t/t-1}}\right)}{(t-1)!}+O\left(f\log^{t-2}\frac{f}{h^{t/t-1}}\right)\\
    &=\frac{f(\log(f)-(t/t-1)\log(h))^{t-1}}{(t-1)!}+O\left(f(\log(f)-(t/t-1)\log(h))^{t-2}\right)\\
    \intertext{$h\ll f$ so $\log(h)<C\log(f)$}
    &=C\frac{f\log^{t-1}(f)}{(t-1)!}+O\left(f\log^{t-2}(f)\right)\\
    &=CZ(t,f)+O\left(f\log^{t-2}(f)\right).\qedhere
  \end{align*}
\end{proof}

\begin{thm}
  \label{thm:sp-count-whp}
  Fix $S$ a proper subset of the variables $\left\{x_1,\ldots,x_n\right\}$ and let $\rmi \sim\umodel$. Then for $f_s,h_{s}$ as in Remark~\ref{remark:hyperbolas}, a constant $C>0$ and arbitrary $\epsilon>0$, the following hold.
  \begin{enumerate}
  \item If $D^{-s}\ll p \ll 1 $ then
    \[\prob{\#\{\text{standard pairs }(x^\alpha,S)\text{ of } \rmi \}=0}\rightarrow 1.\]
  \item If $D^{-n+\epsilon}\leq p\ll D^{-s}$ then
    \[\prob{CZ(t,f_s)<\#\{\text{standard pairs }(x^\alpha,S)\text{ of } \rmi \}<Z(t,h_{s})}\rightarrow 1.\]
  \end{enumerate}

\end{thm}

\begin{proof}\, 

  The first case is clear from the dimension bounds in~\cite[Theorem~3.4]{rmi}, so we focus on the second case.

  \emph{Lower Bound:}
  Notice that if $s=n$, the constraints in the second case are vacuous and so the statement is true vacuously. Thus we will restrict to the case where $s<n$. Now applying Theorem~\ref{thm:std-pair-whp} we find that for any choice of $h_{s+1}$ we have \[\prob{\#\{\text{standard pairs }(x^\alpha,S)\text{ in } I\}\geq |L(f_s,h_{s+1})|}\rightarrow 1.\]

  Since choosing a smaller $f_s$ only weakens the condtion inside the probability, and since $p\ll D^{-s}$, WLOG we may assume $f_{s}\geq D^{-s-\delta}/p$ for any choice of $\delta>0$. Choose $h_{s+1}=D^{-s-1+\delta}/p$. Now we compute 

  \begin{align*}
    \frac{f_{s}^{t-1}}{h_{s+1}^{t}}\geq \frac{\left(D^{-s-\delta}/p\right)^{t-1}}{\left(D^{-s-1+\delta}/p\right)^{t}}
    = \frac{pD^{-st+s-(t-1)\delta}}{D^{-st-t+t\delta}}
    = pD^{n-(2t-1)\delta}
  \end{align*}
  Then since $p\geq D^{-n+\epsilon}$, we can simply choose $\delta<\epsilon/(2t-1)$ to allow $\frac{f_{s}^{t-1}}{h_{s+1}^{t}}\rightarrow \infty$. Then by applying Lemma~\ref{lem:sp-technical} to $L(f_s,h_{s+1})$ we get \[\prob{\#\{\text{standard pairs }(x^\alpha,S)\text{ in } I\}\geq CZ(t,f_s)}\rightarrow 1.\]

  \emph{Upper Bound:}
  If $(x^{\alpha},S)$ is a standard pair then for all $\beta\in \bbZ^{S}$, $x^{\alpha+\beta}\notin I$. Applying Theorem~\ref{thm:projected-hyperbola} implies that if $\prod_{i\notin S}\alpha_i\gg D^{k}$ then $\prob{x^{\alpha+\beta}\notin I, \forall \beta\in \bbZ^{S}}\rightarrow 0$. Thus as an upper bound, \[\prob{\#\{\text{standard pairs }(x^\alpha,S)\text{ in } I\}\leq Z(t,h_s)}\rightarrow 1.\qedhere\]
\end{proof}

Since the arithmetic degree for a monomial ideal is given by the number of standard pairs, the results on standard pairs allow us to compute the asymptotic arithmetic degree of a random monomial ideal.

\begin{corollary}
  Let $\rmi\sim\umodel$, and suppose $p=D^{-k}$ for $0<k<n$. Then there exist constants $C_1,C_2>0$ such that asymptotically almost surely as $D\to\infty$,
  \[C_1Z(t,D^{k-\epsilon})<\adeg(\rmi)<C_2Z(t,D^{k+\epsilon}).\]
\end{corollary}
\begin{proof}
  This is a consequence of applying Theorem~\ref{thm:sp-count-whp} with $S=\emptyset$ and $h_0=D^{k+\epsilon}$ and $f_0=D^{k+\epsilon}$ while noting that there are always asymptotically more standard pairs with $S=\emptyset$ than any other subset of the variables.

\end{proof}

With some refinement, it should be possible to extend the results of Section \ref{sec:standard-pairs} to give us the ``expected Hilbert polynomial" of a random monomial ideal. Since the Hilbert polynomial is a polynomial, it is somewhat less clear what is a reasonable notion of ``expected''. However, for the choice of ``expected'' determined by the simply taking the expected value of the coefficients, we can use the fact that Hilbert polynomial can be expressed as a sum of binomial coefficients indexed over standard pairs of the ideal:

  \[
HP_{S/I}(t) = \sum_{(x^{\alpha},S)\text{ standard pairs of } I}\binom{t-\left|\alpha\right|+\left|S\right|-1}{\left|S\right|-1}.
\]

However, this involves controlling both the number of standard pairs and the size of the monmial part of the standard pair, and so is not directly derivable from our current results.


\section{Acknowledgements}\label{sec:ack}
This material is based upon work supported by the National Science Foundation under Grant No. DMS-1439786 while the first and second authors were in residence at the Institute for Computational and Experimental Research in Mathematics in Providence, RI, during the Fall 2018 Nonlinear Algebra program.

Additionally, this material is based upon work supported by the National Science Foundation under Grant No. DMS-1522158 and DMS-152266 for the first author and Grant No. DMS-1502553 and DMS-1745638 for the third author.

The authors would like to thank Daniel Erman, Jes\'{u}s A. De Loera, and Serkan Ho\c{s}ten for helpful discussions and comments.

\bibliographystyle{siam}

\bibliography{sample.bib}

\end{document}